\documentclass[reqno,12pt]{amsart}
\usepackage{amsfonts}
\usepackage{upgreek}
\usepackage{bbm}
\usepackage{} 
\numberwithin{equation}{section}
\usepackage{indentfirst}
 \usepackage{color}
\usepackage{amssymb}
\usepackage{mathrsfs}
\usepackage{xy}
\usepackage{tikz}
\xyoption{all}

\def\rank{\mbox{\rm rank}\,}
\def\det{\mbox{\rm det}\,}\def\diag{\mbox{\rm diag}\,}

\theoremstyle{plain} 
\newtheorem{theorem}{\bf Theorem}[section]
\newtheorem{lemma}[theorem]{\bf Lemma}
\newtheorem{corollary}[theorem]{\bf Corollary}
\newtheorem{proposition}[theorem]{\bf Proposition}

\theoremstyle{definition} 
\newtheorem{definition}[theorem]{\bf Definition}
\newtheorem{remark}[theorem]{\bf Remark}
\newtheorem{example}[theorem]{\bf Example}

\newcommand{\bt}{\begin{theorem}}
\newcommand{\et}{\end{theorem}}
\newcommand{\bl}{\begin{lemma}}
\newcommand{\el}{\end{lemma}}
\newcommand{\bd}{\begin{definition}}
\newcommand{\ed}{\end{definition}}
\newcommand{\bc}{\begin{corollary}}
\newcommand{\ec}{\end{corollary}}
\newcommand{\bp}{\begin{proof}}
\newcommand{\ep}{\end{proof}}
\newcommand{\bx}{\begin{example}}
\newcommand{\ex}{\end{example}}
\newcommand{\br}{\begin{remark}}
\newcommand{\er}{\end{remark}}
\newcommand{\be}{\begin{equation}}
\newcommand{\ee}{\end{equation}}
\newcommand{\ba}{\begin{align}}
\newcommand{\ea}{\end{align}}
\newcommand{\bn}{\begin{enumerate}}
\newcommand{\en}{\end{enumerate}}
\newcommand{\bcs}{\begin{cases}}
\newcommand{\ecs}{\end{cases}}

\makeatletter
\renewcommand{\section}{\@startsection{section}{1}{0mm}
  {-\baselineskip}{0.5\baselineskip}{\bf\leftline}}
\makeatother

\begin{document}

\title[The smith normal form of the walk matrix of the Dynkin graph $A_n$]{The smith normal form of the walk matrix \\of the Dynkin graph $A_n$} 

\author{Liangwei Huang, Yan Xu and Haicheng Zhang$^*$}
\address{Institute of Mathematics, School of Mathematical Sciences, Nanjing Normal University,
 Nanjing 210023, P. R. China.\endgraf}
\email{2524416777@qq.com (Huang); swgfeng@outlook.com (Xu); zhanghc@njnu.edu.cn (Zhang).}

\thanks{$\ast$: Corresponding author.}
\subjclass[2010]{ 
05C50, 15B36.}
%
\keywords{ 
Smith normal form; Walk matrix; Dynkin graph; Equitable partition.
}



\begin{abstract}
In this paper, we give the rank of the walk matrix of the Dynkin graph $A_n$, and prove that its Smith normal form is $$\diag(\underbrace{1,\ldots,1}_{\lceil\frac{n}{2}\rceil},0,\ldots,0).$$
\end{abstract}

\maketitle

\section{Introduction}
Let $M$ be an $n\times n$ integral matrix. For each $1\leq k\leq n$, the $k$-th determinant divisor of $M$, denoted by $D(k)$, is the greatest common divisor of all $k\times k$ minors of $M$. Suppose that $\rank M=r$, then $D(k)\neq 0$ for $1\leq k\leq r$, and they satisfy $D(k-1)~|~D(k)$ for $k=1,\ldots,r$, where $D(0):=1$. The {invariant factors} of $M$ are defined by
$$d_1=\frac{D(1)}{D(0)}, d_2=\frac{D(2)}{D(1)},\ldots, d_r=\frac{D(r)}{D(r-1)}.$$
It is well known that each $d_i$ divides $d_{i+1}$ for $i=1,\ldots,r-1$. Then the {\em Smith normal form} of $M$ is the diagonal matrix
$$\diag(d_1,\ldots,d_r,0,\ldots,0).$$

Given an $n\times n$ integral matrix $M$, the {\em walk matrix} $W(M)$ of $M$ is defined by
$$W(M)=\left[{e_n}\quad M{e_n}\quad \cdots\quad {M^{n - 1}}{e_n}\right],$$
where $e_n$ is the all-ones vector of dimension $n$. Let $G$ be a simple graph with $n$ vertices and $A$ be the adjacency matrix of $G$. The walk matrix of $A$ is also called the walk matrix of $G$, denoted by $W(G)$. Then the $(i,j)$-entry of the walk matrix $W(G)$ counts the number of the walks in $G$ of length $j-1$ starting from the vertex $i$.

Dynkin graphs and extended Dynkin graphs are widely applied in the study of the classifications of simple Lie algebras in Lie theory (cf. \cite{Cartan,Hum}), the classifications of representation-finite hereditary algebras in representation theory of algebras (cf. \cite{Ass,Deng}), and spectral theory (cf. \cite{Dok}). The Smith normal form of the walk matrix of the Dynkin graph $D_n$ was given in \cite{ch1,Wang}. Recently, S. Moon and S. Park \cite{Moon} have provided formulas for the rank and the Smith normal form of the walk matrix of the extended Dynkin graph $\widetilde{D}_n$.

In this paper, let $n$ be a fixed positive integer, we consider the following Dynkin graph $A_n$ with $n$ vertices
\begin{equation*}\includegraphics[width=0.5\textwidth]{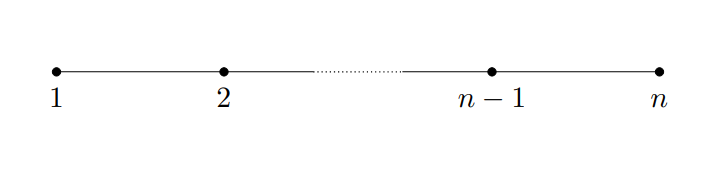}\end{equation*} and give the Smith normal form of its walk matrix.

\section{Preliminaries}
Unless otherwise specified, we set $r=\lceil\frac{n}{2}\rceil$. Let $\overline{W(A_n)}$ be the $r \times r$ matrix obtained from the walk matrix $W(A_n)$ by deleting
\begin{enumerate}
\item[(1)]~the $(r+1)$-th to the $n$-th rows of $W(A_n)$, and
\item[(2)]~the $(r+1)$-th to the $n$-th columns of $W(A_n)$.
\end{enumerate}

Let us give two examples of $W(A_n)$ and $\overline{W(A_n)}$ for $n=3,10$.

\begin{example}
The walk matrix $W({A_3})$ is
$$ \begin{bmatrix}
1&1&2\\
1&2&2\\
1&1&2
\end{bmatrix}$$
and
$$\overline {W({A_3})}  = \begin{bmatrix}
1&1\\
1&2
\end{bmatrix}.$$

The walk matrix $W(A_{10})$ is
$$W({A_{10}}) =  \begin{bmatrix}
1&1&2&3&6&{10}&{20}&{35}&{70}&{126}\\
1&2&3&6&{10}&{20}&{35}&{70}&{126}&{251}\\
1&2&4&7&{14}&{25}&{50}&{91}&{181}&{334}\\
1&2&4&8&{15}&{30}&{56}&{111}&{208}&{409}\\
1&2&4&8&{16}&{31}&{61}&{117}&{228}&{436}\\
1&2&4&8&{16}&{31}&{61}&{117}&{228}&{436}\\
1&2&4&8&{15}&{30}&{56}&{111}&{208}&{409}\\
1&2&4&7&{14}&{25}&{50}&{91}&{181}&{334}\\
1&2&3&6&{10}&{20}&{35}&{70}&{126}&{251}\\
1&1&2&3&6&{10}&{20}&{35}&{70}&{126}
\end{bmatrix}$$
and
$$\overline {W({A_{10}})}= \begin{bmatrix}
1&1&2&3&6\\
1&2&3&6&{10}\\
1&2&4&7&{14}\\
1&2&4&8&{15}\\
1&2&4&8&{16}
\end{bmatrix}.$$
\end{example}

Throughout this paper we freely use the terminology and notation used in \cite{lon}.
We consider two partitions for the vertex set of the Dynkin graph $A_n$. The following partitions with $r$ cells
$$\Pi_1=\{\{1,n\},\{2,n-1\},\ldots,\{\frac{n}{2},\frac{n}{2}+1\}\}$$
and
$$\Pi_2=\{\{1,n\},\{2,n-1\},\ldots,\{\frac{n-1}{2},\frac{n+3}{2}\},\{\frac{n+1}{2}\}\}$$
are equitable for $n$ even and odd, respectively. Note that for any graph $G$, the rank of the walk matrix $W(G)$ is less than or equal to the number of cells in any equitable partitions of the vertex set of $G$ (cf. \cite{lon,same}). Thus, we have $\rank W(A_n)\leq r$.

For each positive integer $m$, we denote by $O_m$ the square zero matrix of order $m$. For any square matrix $M$, set $M\oplus {O_{m}}=\begin{bmatrix}M&\\&{O_{m}}\end{bmatrix}$.
\begin{lemma} \label{ds}
The matrices $W(A_n)$ and $\overline {W({A_n})}\oplus {O_{n - r}}$ have the same Smith normal form.
\end{lemma}
\bp
Let $\rank W(A_n)=t$. Then $t\leq r$. Since the first $t$ columns of $W(A_n)$ are linearly independent over the integer ring $\mathbb{Z}$ (cf. \cite{same,same1}), they form a $\mathbb{Z}$-basis for the column space of $W(A_n)$. Thus the $i$-th column can be written as a $\mathbb{Z}$-linear combination of $${e_n},A{e_n}, \ldots ,{A^{t - 1}}{e_n}$$
for $r+1 \le i \le n$. Note that the $i$-th row and the $(n+1-i)$-th row are identical for $i=1,\ldots,r$. Hence, we obtain $\overline {W({A_n})}\oplus{O_{n - r}}$ by using the elementary row and column operations on $W(A_n)$. Therefore, $W(A_n)$ and $\overline {W({A_n})}\oplus {O_{n - r}}$ have the same Smith normal form.
\ep

Let
$$C_1=\begin{bmatrix}
1&&&\\
&\ddots&&\\
&&1&\\
&&&1\\
&&&1\\
&&1&\\
&{\mathinner{\mkern2mu\raise1pt\hbox{.}\mkern2mu
\raise4pt\hbox{.}\mkern2mu\raise7pt\hbox{.}\mkern1mu}}&&\\
1&&&
\end{bmatrix}
_{n \times r}~\text{and}~{C_2}={\begin{bmatrix}
1&{}&{}&{}\\
{}& \ddots &{}&{}\\
{}&{}&1&{}\\
{}&{}&{}&1\\
{}&{}&1&{}\\
{}& {\mathinner{\mkern2mu\raise1pt\hbox{.}\mkern2mu
 \raise4pt\hbox{.}\mkern2mu\raise7pt\hbox{.}\mkern1mu}} &{}&{}\\
1&{}&{}&{}
\end{bmatrix}_{n \times r}}$$
be the characteristic matrices of the partitions $\Pi_1$ and $\Pi_2$, respectively.
The divisor matrices of $\Pi_1$ and $\Pi_2$ are
\[{B_1} = {\begin{bmatrix}
0&1&{}&{}&{}\\
1&0&1&{}&{}\\
{}&\ddots & \ddots& \ddots &{}\\
{}&{}& 1&0 &1\\
{}&{}&{}&1&1
\end{bmatrix}_{r \times r}}~\text{and}~~{B_2} = {\begin{bmatrix}
0&1&{}&{}&{}\\
1&0&1&{}&{}\\
{}&\ddots & \ddots& \ddots &{}\\
{}&{}& 1&0 &1\\
{}&{}&{}&2&0
\end{bmatrix}_{r \times r}},\]
respectively.

\begin{lemma} \label{xy}
\begin{enumerate}
\item[(a)]~If $n$ is even, then $\overline{W(A_n)}=W(B_1)$.
\item[(b)]~If $n$ is odd, then $\overline{W(A_n)}=W(B_2)$.
\end{enumerate}
\end{lemma}
\bp We only prove (a) and the proof of (b) is similar.
Since $AC_1=C_1B_1$, we have ${A^k}{C_1} = {C_1}B_1^k$ for all $k\geq0$. Since $C_1e_r=e_n$, we obtain
${A^k}{e_n}={A^k}{C_1}{e_r}={C_1}{B_1^k}{e_r}$ for all $k\geq0$. It follows that
$$[{e_n}\quad A{e_n}\quad\cdots\quad{A^{r - 1}}{e_n}] = {C_1}[{e_r}\quad{B_1}{e_r}\quad\cdots\quad B_1^{r - 1}{e_r}].$$
If we delete the $(r+1)$-th to the $n$-th rows of $[{e_n}\quad A{e_n}\quad\cdots\quad{A^{r - 1}}{e_n}]$ and $C_1$, we get $\overline{W(A_n)}$ and the identity matrix, respectively. Hence, $\overline{W(A_n)}=W(B_1)$.
\ep

\begin{lemma}\cite{ch,ch1}\label{wb}
Let $M$ be a real $m\times m$ matrix which is diagonalizable over the real field $\mathbb{R}$. Let $\xi_{1},\xi_{2},\ldots,\xi_{m}$ be $m$ linearly independent eigenvectors of $M^{\rm T}$ corresponding to eigenvalues $\lambda_{1},\lambda_{2},\ldots,\lambda_{m}$, respectively.
Then we have
$${\det}W(M)=\frac{{\prod\nolimits_{1 \le k < j \le m} {({\lambda _j}-{\lambda _k})\prod\nolimits_{j = 1}^m {{e_m^{\rm T}}{\xi _j}} } }}{{\det [{\xi _1}\quad{\xi _2}\quad\cdots\quad{\xi _m}]}}.$$
Moreover, if $\lambda_{1},\lambda_{2},\ldots,\lambda_{m}$ are pairwise different, then \begin{equation*}\rank W(M)=|\{j~|~e_m^{\rm T}\xi_{j}\neq0~\text{and}~j=1,\ldots,m\}|.\end{equation*}
\end{lemma}

\section{Main results}
We are in a position to give the main result of this paper as the following
\bt\label{mainthm}  The walk matrix of the Dynkin graph $A_n$ has the Smith normal form $${\rm diag}(\underbrace{1,\ldots,1}_{\lceil\frac{n}{2}\rceil},0,\ldots,0).$$ \et

Before proving Theorem \ref{mainthm}, we give some preparatory work.
\begin{proposition}\label{ee}
Let $\lambda_{k}=-2\cos \alpha_{k}$ and
$$v_{k}=\begin{bmatrix}
(-1)^{r-1}(1+\sum\limits_{i=1}^{r-1} 2\cos i\alpha_{k}) \\
(-1)^{r-2}(1+\sum\limits_{i=1}^{r-2} 2\cos i\alpha_{k}) \\
\vdots \\
-1-2\cos\alpha_{k} \\
1
\end{bmatrix},$$
where $\alpha_{k}=\frac{2k}{2r+1}\pi$ for $k=1,\ldots,r$. Then $v_{k}$ is an eigenvector of $B_{1}^{\rm T}$ corresponding to the eigenvalue $\lambda_{k}$ for each $k=1,\ldots,r$.
\end{proposition}
\begin{proof}
For any $k=1,\ldots,r$, firstly, we show that the first entries of $B_{1}^{\rm T}v_{k}$ and $\lambda_{k}v_{k}$ are the same, i.e.
\begin{equation}\label{LR}
(-1)^{r-2}(1+\sum\limits_{i=1}^{r-2}2\cos i\alpha_{k}) = -2\cos \alpha_{k}(-1)^{r-1}(1+\sum\limits_{i=1}^{r-1} 2\cos i\alpha_{k}).\end{equation}

Let us calculate the right hand side of (\ref{LR})
$${\rm RHS}~\text{of}~(\ref{LR})= 2(-1)^{r-2}(\cos \alpha_{k}+\sum\limits_{i=1}^{r-1} 2\cos i\alpha_{k}\cos \alpha_{k}).$$
By applying the formula $2\cos\alpha \cos\beta = \cos (\alpha + \beta) + \cos (\alpha - \beta)$, we have
\begin{equation}
\begin{split}
{\rm RHS}~\text{of}~(\ref{LR})&=2(-1)^{r-2}(\cos \alpha_{k}+\sum\limits_{i=1}^{r-1} (\cos (i+1)\alpha_{k}+\cos (i-1)\alpha_{k})) \\
&=2(-1)^{r-2}(\sum\limits_{i=1}^{r} \cos i\alpha_{k}+1+\sum\limits_{i=1}^{r-2} \cos i\alpha_{k})\label{a}.
\end{split}
\end{equation}
Note that $\sin\alpha\cos\beta = \frac{1}{2}(\sin(\beta+\alpha)-\sin(\beta-\alpha))$, and
\begin{equation}\label{js1}
\begin{split}\sin\frac{1}{2}\alpha_{k}\sum\limits_{i=1}^{r}\cos i\alpha_{k} &= \frac{1}{2}\sum\limits_{i=1}^{r}(\sin(i+\frac{1}{2})\alpha_{k}-\sin(i-\frac{1}{2})\alpha_{k})\\
&= \frac{1}{2}(\sin(r+\frac{1}{2})\alpha_{k}-\sin\frac{1}{2}\alpha_{k}).
\end{split}
\end{equation}
Since $\frac{1}{2}\alpha_{k}=\frac{k}{2r+1}\pi\in(0,\frac{\pi}{2})$, we have $\sin\frac{1}{2}\alpha_{k}\neq0$ and $\sin(r+\frac{1}{2})\alpha_{k}=\sin k\pi=0$. Thus, by (\ref{js1}), we get
\begin{equation}\label{b}
\sum\limits_{i=1}^{r}\cos i\alpha_{k}=-\frac{1}{2}.
\end{equation}
Substituting (\ref{b}) into (\ref{a}), we obtain
$${\rm RHS}~\text{of}~(\ref{LR}) = 2(-1)^{r-2}(\frac{1}{2}+\sum\limits_{i=1}^{r-2}\cos i\alpha_{k})={\rm LHS}~\text{of}~(\ref{LR}).$$

Now, we prove that the $j$-th entries of $B_{1}^{\rm T}v_{k}$ and $\lambda_{k}v_{k}$ are equal for $j=2,\ldots,r-1$. In fact, using $\cos\alpha+\cos\beta=2\cos \frac{\alpha+\beta}{2}\cos\frac{\alpha-\beta}{2}$, we obtain that the $j$-th entry of $B_{1}^{\rm T}v_{k}$ is
\begin{equation*}
\begin{split}
&(-1)^{r-j+1}(1+\sum\limits_{i=1}^{r-j+1}2\cos i\alpha_{k})+(-1)^{r-j-1}(1+\sum\limits_{i=1}^{r-j-1}2\cos i\alpha_{k}) \\
&=(-1)^{r-j+1}(2+2\cos\alpha_{k}+2\cos 2\alpha_{k}+\sum\limits_{i=2}^{r-j}2(\cos(i+1)\alpha_{k}+\cos(i-1)\alpha_{k})) \\
&=(-1)^{r-j+1}(2\cos\alpha_{k}+4\cos\alpha_{k}^{2}+\sum\limits_{i=2}^{r-j}4\cos i\alpha_{k}\cos\alpha_{k})\\
&=-2\cos\alpha_{k}(-1)^{r-j}(1+\sum\limits_{i=1}^{r-j}2\cos i\alpha_{k}).
\end{split}
\end{equation*}

By
$$-1-2\cos \alpha_{k}+1=-2\cos\alpha_{k},$$ clearly, the last entries of $B_{1}^{\rm T}v_{k}$ and $\lambda_{k}v_{k}$ are also equal. Therefore, $B_{1}^{\rm T}v_{k}=\lambda_{k}v_{k}$.
\end{proof}

\begin{lemma}\label{ev}
Let $e_{r}$ be the all-ones vector of dimension $r$, and $v_{k}$ be as above. Then $$\prod\limits_{k=1}^{r}e_{r}^{\rm T}v_{k}=(-1)^{\lfloor\frac{r}{2}\rfloor}.$$
\end{lemma}
\begin{proof}
Firstly, we show that
\begin{equation}\label{c}
\prod\limits_{k=1}^{r}e_{r}^{\rm T}v_{k}=\prod\limits_{k=1}^{r}\frac{\sin r\alpha_{k}}{\sin\alpha_{k}}.
\end{equation}
If $r$ is odd, then
$$e_{r}^{\rm T}v_{k}=1+2\cos2\alpha_{k}+2\cos4\alpha_{k}+\cdots+2\cos(r-3)\alpha_{k}+2\cos(r-1)\alpha_{k}.$$
By applying the formula $2\cos\alpha\sin\beta=\sin(\alpha+\beta)-\sin(\alpha-\beta)$, we have
\begin{equation*}
\begin{split}
&(1+2\cos2\alpha_{k}+2\cos4\alpha_{k}+\cdots+2\cos(r-3)\alpha_{k}+2\cos(r-1)\alpha_{k})\sin\alpha_{k}\\
&=\sin\alpha_{k}+(\sin3\alpha_{k}-\sin\alpha_{k})+\cdots+(\sin r\alpha_{k}-\sin(r-2)\alpha_{k}) \\
&=\sin r\alpha_{k}.
\end{split}
\end{equation*}
Note that $\alpha_{k}=\frac{2k}{2r+1}\pi\in(0,\pi)$, so $\sin\alpha_{k}\neq0$. Hence, Equation (\ref{c}) holds.

If $r$ is even, then $$e_{r}^{\rm T}v_{k}=-2\cos\alpha_{k}-2\cos3\alpha_{k}-\cdots-2\cos(r-3)\alpha_{k}-2\cos(r-1)\alpha_{k}.$$
Similarly, $$(-2\cos\alpha_{k}-2\cos3\alpha_{k}-\cdots-2\cos(r-3)\alpha_{k}-2\cos(r-1)\alpha_{k})\sin\alpha_{k}=-\sin r\alpha_{k}.$$
In this case, we have $$\prod\limits_{k=1}^{r}e_{r}^{\rm T}v_{k}=(-1)^{r}\prod\limits_{k=1}^{r}\frac{\sin r\alpha_{k}}{\sin\alpha_{k}}=\prod\limits_{k=1}^{r}\frac{\sin r\alpha_{k}}{\sin\alpha_{k}}.$$
Hence, Equation (\ref{c}) holds.

Note that
\begin{equation*}r\alpha_{k}=k\pi-\frac{k}{2r+1}\pi~\text{and}~r\alpha_{k}=(k-1)\pi+\frac{2r-k+1}{2r+1}\pi.\end{equation*}

Suppose that $r$ is odd. When $k$ is odd, we have $$\sin r\alpha_{k}=\sin\frac{2r-k+1}{2r+1}\pi$$ and $\{2r-k+1~|~k=1,3,\ldots,r\}=\{r+1,r+3,\ldots,2r\}$.
When $k$ is even, we have $$\sin r\alpha_{k}=-\sin\frac{k}{2r+1}\pi,$$ where $k=2,4,\ldots,r-1$. Thus, in this case,
$$\prod\limits_{k=1}^{r}\sin r\alpha_{k}=(-1)^{\lfloor\frac{r}{2}\rfloor}\prod\limits_{k=1}^{r}\sin\frac{2k}{2r+1}\pi=(-1)^{\lfloor\frac{r}{2}\rfloor}\prod\limits_{k=1}^{r}\sin\alpha_{k}.$$

Suppose that $r$ is even. When $k$ is odd, we have $$\sin r\alpha_{k}=\sin\frac{2r-k+1}{2r+1}\pi$$ and $\{2r-k+1~|~k=1,3,\ldots,r-1\}=\{r+2,r+4,\ldots,2r\}$.
When $k$ is even, we have $$\sin r\alpha_{k}=-\sin\frac{k}{2r+1}\pi,$$ where $k=2,4,\ldots,r$. Thus, we also have
$$\prod\limits_{k=1}^{r}\sin r\alpha_{k}=(-1)^{\lfloor\frac{r}{2}\rfloor}\prod\limits_{k=1}^{r}\sin\alpha_{k}.$$
Therefore, we obtain $\prod\limits_{k=1}^{r}e_{r}^{\rm T}v_{k}=(-1)^{\lfloor\frac{r}{2}\rfloor}$, and complete the proof.
\end{proof}

The following lemma is needed for calculating $\det\left[v_{1}\quad v_{2}\quad \cdots\quad v_{r}\right]$.
\begin{lemma}\label{cc}\cite[Lemma 11]{Wang}
It holds that
$$\begin{vmatrix}
1 & 1 & \cdots & 1 \\
2\cos\theta_{1} & 2\cos\theta_{2} & \cdots & 2\cos\theta_{q} \\
2\cos2\theta_{1} & 2\cos2\theta_{2} & \cdots & 2\cos2\theta_{q} \\
\vdots & \vdots &  & \vdots \\
2\cos(q-1)\theta_{1} & 2\cos(q-1)\theta_{2} & \cdots & 2\cos(q-1)\theta_{q}
\end{vmatrix}=\prod\limits_{1\leq j<i\leq q}(2\cos\theta_{i}-2\cos\theta_{j}).$$
\end{lemma}

\begin{proposition}\label{det}
Suppose that $n$ is even. Then the determinant of $W(B_{1})$ is
$$\det W(B_{1})=1.$$
\end{proposition}
\begin{proof}
Firstly, let us calculate the determinant of the matrix $\left[v_{1}\quad v_{2}\quad \cdots\quad v_{r}\right]$,
where each $v_{k}$ is the same as that in Proposition \ref{ee}. For each $i=1,\ldots,r-1$, adding the $(i+1)$-th row to the $i$-th row, we obtain
\begin{equation*}
\begin{split}
&\det\begin{bmatrix}
v_{1} & v_{2} & \cdots & v_{r}
\end{bmatrix} \\
&=\begin{vmatrix}
(-1)^{r-1}(1+\sum\limits_{i=1}^{r-1}2\cos i\alpha_{1})&(-1)^{r-1}(1+\sum\limits_{i=1}^{r-1}2\cos i\alpha_{2})& \cdots &(-1)^{r-1}(1+\sum\limits_{i=1}^{r-1}2\cos i\alpha_{r})\\
(-1)^{r-2}(1+\sum\limits_{i=1}^{r-2}2\cos i\alpha_{1})&(-1)^{r-2}(1+\sum\limits_{i=1}^{r-2}2\cos i\alpha_{2})& \cdots &(-1)^{r-2}(1+\sum\limits_{i=1}^{r-2}2\cos i\alpha_{r})\\
\vdots & \vdots & & \vdots \\
-1-2\cos\alpha_{1}&-1-2\cos\alpha_{2}& \cdots &-1-2\cos\alpha_{r}\\
1 & 1 & \cdots & 1
\end{vmatrix} \\
&=\begin{vmatrix}
(-1)^{r-1}2\cos(r-1)\alpha_{1}&(-1)^{r-1}2\cos(r-1)\alpha_{2}& \cdots &(-1)^{r-1}2\cos(r-1)\alpha_{r}\\
(-1)^{r-2}2\cos(r-2)\alpha_{1}&(-1)^{r-2}2\cos(r-2)\alpha_{2}& \cdots &(-1)^{r-2}2\cos(r-2)\alpha_{r}\\
\vdots & \vdots & & \vdots \\
-2\cos\alpha_{1}&-2\cos\alpha_{2}& \cdots &-2\cos\alpha_{r}\\
1 & 1 & \cdots & 1
\end{vmatrix} \\
&=(-1)^{\frac{(r-1)r}{2}}\begin{vmatrix}
2\cos(r-1)\alpha_{1} & 2\cos(r-1)\alpha_{2} & \cdots & 2\cos(r-1)\alpha_{r} \\
2\cos(r-2)\alpha_{1} & 2\cos(r-2)\alpha_{2} & \cdots & 2\cos(r-2)\alpha_{r} \\
\vdots & \vdots & & \vdots \\
2\cos\alpha_{1}&2\cos\alpha_{2}&\cdots&2\cos\alpha_{r} \\
1 & 1 &\cdots& 1
\end{vmatrix} \\
&=(-1)^{\frac{(r-1)r}{2}+\lfloor\frac{r}{2}\rfloor}\begin{vmatrix}1 & 1 &\cdots& 1\\
2\cos\alpha_{1}&2\cos\alpha_{2}&\cdots&2\cos\alpha_{r} \\\vdots & \vdots & & \vdots \\2\cos(r-2)\alpha_{1} & 2\cos(r-2)\alpha_{2} & \cdots & 2\cos(r-2)\alpha_{r} \\
2\cos(r-1)\alpha_{1} & 2\cos(r-1)\alpha_{2} & \cdots & 2\cos(r-1)\alpha_{r}
\end{vmatrix}.
\end{split}
\end{equation*}
Using Lemma \ref{cc}, we get
\begin{equation*}
\begin{split}
\det\begin{bmatrix}v_{1} & v_{2} & \cdots & v_{r}
\end{bmatrix}&=(-1)^{\frac{(r-1)r}{2}+\lfloor\frac{r}{2}\rfloor}\prod\limits_{1\leq j<i\leq r}(2\cos\alpha_{i}-2\cos\alpha_{j})\\
&=(-1)^{\lfloor\frac{r}{2}\rfloor}\prod\limits_{1\leq j<i\leq r}(-2\cos\alpha_{i}+2\cos\alpha_{j})\\
&=(-1)^{\lfloor\frac{r}{2}\rfloor}\prod\limits_{1\leq j<i\leq r}(\lambda_{i}-\lambda_{j}),
\end{split}
\end{equation*}
where we recall that $\lambda_k:=-2\cos\alpha_{k}$.
By Lemma \ref{wb} and Lemma \ref{ev}, we complete the proof.
\end{proof}

\begin{proposition}\label{ee1}
Let $\mu_{k}=2\cos\beta_{k}$ and $$w_{k}=\begin{bmatrix}
2\cos(r-1)\beta_{k} \\
2\cos(r-2)\beta_{k} \\
\vdots \\
2\cos\beta_{k} \\
1
\end{bmatrix},$$
where $\beta_{k}=\frac{2k-1}{2r}\pi$ for $k=1,\ldots,r$. Then $w_{k}$ is an eigenvector of $B_{2}^{\rm T}$ corresponding to the eigenvalue $\mu_{k}$ for each $k=1,\ldots,r$.
\end{proposition}
\begin{proof}
By the formula $\cos\alpha+\cos\beta=2\cos\frac{\alpha+\beta}{2}\cos\frac{\alpha-\beta}{2}$, we have $$\cos(i-1)\beta_{k}+\cos(i+1)\beta_{k}=2\cos i\beta_{k}\cos\beta_{k}$$
for each $i=2,\ldots,r-1$. Noting that $\cos r\beta_{k}=\cos\frac{2k-1}{2}\pi=0$, we have $$2\cos (r-2)\beta_{k}=2\cos (r-2)\beta_{k}+2\cos r\beta_{k}=4\cos (r-1)\beta_{k}\cos \beta_{k}.$$
Then it is easy to see
$$\begin{bmatrix}
0 & 1 &   &   &   \\
1 & 0 & 1 &   &   \\
& \ddots & \ddots & \ddots &   \\
&   & 1 & 0 & 2 \\
&   &   & 1 & 0
\end{bmatrix}
\begin{bmatrix}
2\cos(r-1)\beta_{k} \\
2\cos(r-2)\beta_{k} \\
\vdots \\
2\cos\beta_{k} \\
1
\end{bmatrix}
=2\cos\beta_{k}
\begin{bmatrix}
2\cos(r-1)\beta_{k} \\
2\cos(r-2)\beta_{k} \\
\vdots \\
2\cos\beta_{k} \\
1
\end{bmatrix}.$$ Hence, $B_{2}^{\rm T}w_{k}=\mu_{k}w_{k}$.
\end{proof}

\begin{lemma}\label{ev1}
Let $e_{r}$ be the all-ones vector of dimension $r$, and $w_{k}$ be as above. Then $$\prod\limits_{k=1}^{r}e_{r}^{\rm T}w_{k}=(-1)^{\lfloor\frac{r}{2}\rfloor}.$$
\end{lemma}
\begin{proof}
By the formula $2\cos\alpha\sin\beta=\sin(\alpha+\beta)-\sin(\alpha-\beta)$, we obtain
\begin{equation*}
\begin{split}
&(1+2\cos\beta_{k}+\cdots+2\cos(r-1)\beta_{k})\sin\frac{1}{2}\beta_{k}\\
&=\sin(r-\frac{1}{2})\beta_{k}
=\sin(\frac{2k-1}{2}\pi-\frac{2k-1}{4r}\pi) \\
&=\pm\cos\frac{2k-1}{4r}\pi
=\pm\sin(\frac{\pi}{2}-\frac{2k-1}{4r}\pi) \\
&=\pm\sin\frac{2r-2k+1}{4r}\pi.
\end{split}
\end{equation*}
In fact, it is easy to see that the cardinality of the set
\begin{equation*}
\begin{split}
&\{k~|~\sin(\frac{2k-1}{2}\pi-\frac{2k-1}{4r}\pi)=-\cos\frac{2k-1}{4r}\pi,~k=1,\ldots,r\}\\
&=\{k~|~k=1,\ldots,r~\text{and}~k~\text{is~even}\}
\end{split}
\end{equation*}
is equal to $\lfloor\frac{r}{2}\rfloor$.
Note that $\{2r-2k+1~|~k=1,\ldots,r\}=\{1,3,\ldots,2r-1\}$ and $\sin\frac{1}{2}\beta_{k}=\sin\frac{2k-1}{4r}\pi\neq0$. Hence, we obtain
$$\prod\limits_{k=1}^{r}e_{r}^{\rm T}w_{k}=(-1)^{\lfloor\frac{r}{2}\rfloor}\frac{\prod\limits_{k=1}^{r}\sin\frac{2r-2k+1}{4r}\pi}{\prod\limits_{k=1}^{r}\sin\frac{2k-1}{4r}\pi}=(-1)^{\lfloor\frac{r}{2}\rfloor}.$$
\end{proof}

\begin{proposition}\label{det1}
Suppose that $n$ is odd. Then the determinant of $W(B_{2})$ is $$\det W(B_{2})=1.$$
\end{proposition}
\begin{proof}
Firstly, let us calculate the determinant of the matrix $\left[w_{1}\quad w_{2}\quad \cdots\quad w_{r}\right]$,
where each $w_{k}$ is the same as that in Proposition \ref{ee1}.
\begin{equation*}
\begin{split}
&\det\begin{bmatrix}
w_{1} & w_{2} & \cdots & w_{r}
\end{bmatrix} \\
=&\begin{vmatrix}
2\cos(r-1)\beta_{1}&2\cos(r-1)\beta_{2}& \cdots &2\cos(r-1)\beta_{r}\\
2\cos(r-2)\beta_{1}&2\cos(r-2)\beta_{2}& \cdots &2\cos(r-2)\beta_{r}\\
\vdots & \vdots & & \vdots \\
2\cos\beta_{1}&2\cos\beta_{2}& \cdots &2\cos\beta_{r}\\
1 & 1 & \cdots & 1
\end{vmatrix} \\
=&(-1)^{\lfloor\frac{r}{2}\rfloor}\begin{vmatrix}1 & 1 &\cdots& 1\\
2\cos\beta_{1}&2\cos\beta_{2}& \cdots &2\cos\beta_{r}\\
\vdots & \vdots & & \vdots \\
2\cos(r-2)\beta_{1}&2\cos(r-2)\beta_{2}& \cdots &2\cos(r-2)\beta_{r}\\
2\cos(r-1)\beta_{1}&2\cos(r-1)\beta_{2}& \cdots &2\cos(r-1)\beta_{r}
\end{vmatrix}\\
=&(-1)^{\lfloor\frac{r}{2}\rfloor}\prod\limits_{1\leq j<i\leq r}(2\cos\beta_{i}-2\cos\beta_{j})\\
=&(-1)^{\lfloor\frac{r}{2}\rfloor}\prod\limits_{1\leq j<i\leq r}(\mu_{i}-\mu_{j}).
\end{split}
\end{equation*}
Using Lemma \ref{wb} and Lemma \ref{ev1}, we obtain $\det W(B_{2})=1$.
\end{proof}

\textbf{\em{Proof of Theorem \ref{mainthm}:}}
\bp When $n$ is even, by Lemma \ref{xy}, we have $$\det\overline{W(A_n)}=\det W(B_1)=1.$$ Thus, by the definition of invariant factors and Lemma \ref{ds}, the Smith normal form of $W(A_n)$ is $${\rm diag}(\underbrace{1,\ldots,1}_{r=\lceil\frac{n}{2}\rceil},0,\ldots,0).$$
The proof is similar when $n$ is odd. Therefore, we complete the proof.
\ep

At the end of the paper, let us provide a short and direct proof of $\det~\overline{W(A_n)}=1$, which is suggested by the anonymous referee.

For each $1\leq i, j \leq n$,
let $w_{i,j}$ be the $(i,j)$-entry of $W(A_n)$, which counts the number of the walks in $A_n$ of length $j-1$ starting from the vertex $i$. Clearly, we have $w_{i,1}=1$ for any $1\leq i\leq n$.

For $1\leq i\leq n$ and $2\leq j\leq n$, set
$$W_{i,j}=\{(x_1,x_2,\ldots,x_{j-1})|x_k=\pm1~\text{for}~1\leq k \leq j-1,~\text{and}~1\leq i+\sum\limits_{k=1}^{s-1}x_k\leq n~\text{for}~2 \leq s\leq j\}.$$
If we regard each walk $\rho$ in $A_n$ of length $j-1$ as a sequence $(x_1,x_2,\ldots,x_{j-1})$ consisting of $\pm 1$, and $x_k=1$ (resp. $x_k=-1$) means that the $k$-th step of the walk $\rho$ points to the right (resp. left), then $W_{i,j}$ denotes
the set of the walks in $A_n$ of length $j-1$ starting from the vertex $i$. For example, for $(-1,-1,1)\in W_{3,4}$, it denotes the walk in $A_n~(n\geq4)$ starting from the vertex $3$, which firstly proceeds two steps to the left and then one step to the right. Thus, $|W_{i,j}|=w_{i,j}$ for any $1\leq i\leq n$ and $2\leq j\leq n$.

\begin{lemma}\label{wij}
For any $2\leq i\leq \lceil \frac{n}{2}\rceil$, we have

$(1)$ $w_{i,j}=w_{i-1,j}$ for $1\leq j\leq i-1$; and

$(2)$ $w_{i,i}=w_{i-1,i}+1$.
\end{lemma}

\bp For any $2\leq i\leq \lceil \frac{n}{2}\rceil$, clearly, $w_{i,1}=w_{i-1,1}=1$.
If $2\leq j\leq{\rm min}\{i,n-i+1\}$, then for any sequence $(x_1,x_2,\ldots,x_{j-1})$~with each component equal to $\pm1$, we have
$$1\leq i+(j-1)(-1)\leq i+\sum\limits_{i=1}^{s-1}x_k\leq i+(j-1)\leq n$$
for $2\leq s\leq j$. Thus,
$$W_{i,j}=\{(x_1,x_2,\ldots,x_{j-1})~|~x_k=\pm1~\text{for}~1\leq k \leq j-1\}.$$
Hence, $$w_{i,j}=|W_{i,j}|=2^{j-1}.$$ In particular, if $2\leq i\leq\lceil\frac{n}{2}\rceil$, then $i\leq n-i+1$. Thus $w_{i,j}=2^{j-1}=w_{i-1,j}$ for $2\leq j\leq i-1$ and $w_{i,i}=2^{i-1}$.

Note that $i-1+\sum\limits_{k=1}^{i-1}(-1)=0<1$, thus $(\underbrace{-1,\ldots,-1}_{i-1})\notin W_{i-1,i}$. However, for any sequence $(x_1,x_2,\ldots,x_{i-1})$~with each component equal to $\pm1$, if it is not the sequence $(\underbrace{-1,\ldots,-1}_{i-1})$, then we have $$1= i-1+(i-2)(-1)\leq i-1+\sum\limits_{k=1}^{s-1}x_k\leq i-1+i-1\leq n-1$$
for $2\leq s\leq i$.
Namely, $(x_1,x_2,\ldots,x_{i-1})\in W_{i-1,i}$. Hence, $w_{i-1,i}=2^{i-1}-1$. Therefore, we complete the proof.
\ep

Using Lemma \ref{wij} and taking the elementary row operations $r_i-r_{i-1}$ for $i=n,\ldots,2$, we get that
$$\det~\overline{W(A_n)}=\begin{vmatrix}
1 & \ast   & \ast  & \cdots & \ast & \ast \\
0 & 1 & \ast   & \cdots & \ast & \ast \\
0 & 0 & 1 & \cdots & \ast & \ast \\
\vdots & \vdots & \vdots &  & \vdots & \vdots \\
0 & 0 & 0 & \cdots & 1 & \ast \\
0 & 0 & 0 & \cdots & 0 & 1
\end{vmatrix}=1.$$

\section*{Acknowledgments}
The authors are grateful to the anonymous referee for the careful reading, valuable comments and suggestions.
This work was partially supported by the National Natural Science Foundation of China (Grant No. 12271257).

\end{document}